\documentclass[reqno,12pt,a4paper]{amsart}
\usepackage{amsmath,amssymb,exscale, amscd}
\usepackage[utf8]{inputenc}
\usepackage{amsmath}
\usepackage[all]{xy}
\usepackage[english]{babel}
\usepackage{mathrsfs}
\usepackage{mathalfa}
\usepackage{graphicx}
\usepackage{tikz}
\usepackage{epigraph}
\usepackage{setspace}
\usepackage{enumitem}
\usepackage[utf8]{inputenc}
\usepackage{float}
\usepackage{amsthm}
\usepackage{amssymb}
\usepackage{a4wide}
\usepackage{mathtools}

\usepackage{tikz-cd}

\usepackage[symbol]{footmisc}

\usepackage{comment}

\usepackage{hyperref}
                 \hypersetup{ pdfborder={0 0 0}, 
                              colorlinks=true, 
                              citecolor=blue,
                              linktoc=page,
                              pdfauthor={}, 
                              pdftitle={}}
\renewcommand{\ref}{\hyperref}

\newtheorem{theorem}{Theorem}[section]
\newtheorem{lemma}[theorem]{Lemma}
\newtheorem{corollary}[theorem]{Corollary}
\newtheorem{proposition}[theorem]{Proposition}

\theoremstyle{definition}

\newtheorem{remark}[theorem]{Remark}
\newtheorem{remarks}[theorem]{Remarks}
\newtheorem{example}[theorem]{Example}

\theoremstyle{plain}

\def\L{{\mathcal L}}
\def\l{{\ell}}
\def\pp{\mathbb P}
\def\ff{\mathbb F}
\def\cc{\mathbb C}
\def\Fd{{{\rm F}_d}}
\def\Z{\mathbb{Z}}
\def\rc{{\rm C}}
\usepackage{verbatim}

\newcommand\restr[2]{{
  \left.\kern-\nulldelimiterspace 
  #1 
  \littletaller 
  \right|_{#2} 
  }}

\newcommand{\littletaller}{\mathchoice{\vphantom{\big|}}{}{}{}}

\makeatletter
\newcommand{\extp}{\@ifnextchar^\@extp{\@extp^{\,}}}
\def\@extp^#1{\mathop{\bigwedge\nolimits^{\!#1}}}
\makeatother
\font\smallrm=cmr7

\author[
\smallrm{S.~Andria, J.~Rojas and W.~Mangueira}
]
{Sally Andria, Jacqueline Rojas and W\'allace Mangueira}
\title[Skew lines]{Maximal number of Skew lines on Fermat Surfaces}


\begin{document}

\begin{abstract} It is well-known that the Fermat surface of degree $d\geq 3$ has $3d^2$ lines. However, it has not yet been established what is the maximal number of pairwise disjoint lines that it can have if $d\geq 4$. In this article we show that the maximal number of skew lines on the Fermat surface of degree $d\geq 4$ is $3d$, either $d$ even or $d$ odd distinct of 5, otherwise ($d=5$) it contains no more than 13 pairwise disjoint lines.
\end{abstract}

\maketitle

\section*{Introduction}
	\quad It is well-known that the Fermat surface of degree $d$ in the complex projective space has $3d^2$ lines for $d\geq 3$, so it is a lower bound for $\ell_d$, the  maximal number of lines  that a smooth surface of degree $d$ in $\pp^3$ can have (cf. Proposition \ref{prop:3d2lines}). In fact, since 1882 it has been know that the so called Schur's quartic contains exactly 64 lines (\cite{Schur}). And only in 1943, {\it B. Segre} proved that $\ell_4=64$ (\cite{Segre})\footnote{ Even though a gap was discovered in Segre's proof by {\it Rams\--Sch\"utt} in 2015 (\cite{RamsSchutt}), the claim is still correct.}, but $\ell_d$ remains unknown for $d \geq  5$. 
 In this regard, the articles of {\it Caporaso\--Harris\--Mazur} (\cite{CaporasoHarrisMazur}) and {\it Boissi\`ere\--Sarti} (\cite{BoissiereSarti}) exhibited lower bounds for these numbers, which leads us to infer that $3d^2$ does not provide the maximal number of lines  on a smooth  surface of  degree $d\geq 4$ in characteristic $0$. On the other hand, according the Bauer-Rams $11d^2-30d+18$ is an upper bound for the maximal number of lines  on a smooth  surface of  degree $d\geq 3$ in $\pp^3(k)$ being $k$ a field of characteristic 0 or of characteristic $p > d$ (\cite{BauerRams}).  For example, the Fermat surface, defined by the vanishing of the polynomial $x^{q+1} + y^{q+1} +
	z^{q+1 }+ w^{q+1}$ on $\pp^3(k)$ being $k$  a field extension of $\ff_{q^2}$ where $q = p^e$ for a prime $p$, contains $q^4+q^3+q+ 1$ lines, which exceeds the Bauer-Rams's upper bound and leads the authors (cf. \cite{BrosowskyEtAll} and  references there in) to conjecture that  these Fermat surfaces may provide the maximal number of lines possible on a surface of a given degree in characteristic $p > 0$.  

	Another problem related to this is to determine the maximal number, ${\frak s}_d$, of pairwise disjoint lines (or skew lines) that a smooth surface of degree $d$ can have.  In 1975, {\it Miyaoka} gave the upper bound ${\frak s}_d\leq 2d(d-2)$ if $d\geq 4$ (\cite{Miyaoka}). It is known that ${\frak s}_3=6$, ${\frak s}_4=16$ (\cite{Nikulin})  and ${\frak s}_6=48$ (\cite{FerreiraLiraRojas}).
	Some lower bounds were given by {\it  Rams} (\cite{Rams}) and {\it Boissiere\--Sarti} (\cite{BoissiereSarti}).  However, ${\frak s}_d$ remains unknown for  $d=5$ and $d\geq 7$.


To the best of our knowledge, the maximal number of pairwise disjoint lines on Fermat surfaces is not explicitly stated in the modern literature. For instance, in Rams' article (\cite{Rams}), it is mentioned: “Let us note that the Fermat surface $\Fd$, i.e., the surface with $3d^2$ lines (the largest number known so far for $d\neq 4, 6, 8, 12, 20$), contains no family of $3d$ pairwise disjoint lines” but this claim is made without proof. For Fermat surfaces over fields with characteristic $p \neq 0$, \cite{BrosowskyEtAll} provides certain bounds for $p=2,3$.

The aim of our work is to show, in an elementary and self-contained way, that the maximal number of pairwise disjoint lines on Fermat surfaces of degree $d\geq 3$ over the complex numbers is exactly $3d$ for any $d$ even and $d$ odd distinct from 3 and 5 (being such numbers 6 and 13, for $d=3,5$, respectively), according to Theorem~\ref{Teo}.

In order to do that we first  established a notation for the set of lines in $\Fd$ (see  (\ref{retasL0L1L2})), in such a way that, we obtain the  stratification 
$\L^0\,\dot\cup\,\L^1\,\dot\cup\,\L^2$ with $\#\L^i=d^2,$ for $i=0,1,2$ of these lines in $\Fd$ (cf. Proposition~\ref{prop:3d2lines}). Moreover, the relations established in Proposition~\ref{PropIntersecoesLs}, together with Proposition~\ref{PropPsiPhis}  give us enough conditions to study the intersection  between the lines on  families $\L^i$ and $\L^j$ for $i\neq j$.  Next, we check that the   maximal number of pairwise disjoint lines on the family $\L^i$  is $d$ for all $i$, which implies   that ${\frak s}(\Fd)\leq 3d$ (being ${\frak s}(\Fd)$ the maximal number of pairwise disjoint lines that $\Fd$ can have).  In fact, if $d$ is even,  then we easily get a family consisting of $3d$ pairwise disjoint lines  on $\Fd$ (cf. Proposition~\ref{Prop:dpar}), otherwise we are faced with a real/generalized `Sudoku game' to find such maximal set of pairwise disjoint lines on $\Fd$ (cf. Sections~\ref{Secd35}, \ref{Secdgeq4}). To our surprise the case $d=5$ was the only one (for $d\geq 4$) that there is no family with $3d$ skew lines.

Finally, we note that to study the maximal number of rational curves (in particular lines) which do not intersect one another on a surface is an important tool to classify surfaces in the projective space (cf. \cite{Nikulin}, \cite{Bauer} and \cite{ArmstrongPoveroSalamon}), as well as to determine all the lines on a smooth surface from some set of its skew lines (\cite{MckeanMinahanZhang}).
\section{Lines on Fermat surfaces}

Let ${\rm F}_d$ be the degree $d$ Fermat surface in the projective complex space defined as the zeros locus of $$ x^d-y^d-z^d+w^d\in \mathbb{C}[x,y,z,w].$$ 
Set $\Phi(\Fd)=\{\ell \subset \Fd \ | \ \ell \textrm{ is a line}\}$. An easy verification allows us to conclude that $\mathcal{L}^j=\{L_{k,i}^j\}\subset \Phi(\Fd)$ for $j=0,1,2$ being
\begin{equation}\label{retasL0L1L2}
L^0_{k,i} :\left\{\begin{array}{l}
y = \eta^{i} x\\
w = \eta^k z\end{array}\right.
\!\!\!\!,\,\,\,\,
L^1_{k,i} :\left\{\begin{array}{l}
x = \eta^{k+i} z\\
y = \eta^i w\end{array}\right.
\,\,\hbox{ and }\quad
L^2_{k,i} :\left\{\begin{array}{l}
x = v\eta^iw\\
y = v\eta^{k+i} z\end{array}\right.
\end{equation} 
 where $\eta$ is a primitive $d$th root of the unity, $v$ is a complex number such that $v^d=-1$ 
 and $k,i\in \{0,1,2,\ldots,d-1\}$.\footnote{Here we use the indices $i$ and $k+i$ to describe the lines on families ${\L}^1$and ${\L}^2$ instead of simply $i,k$, because this simplifies the writing of incidence relations between the lines in $\Fd$, as we will see later.} Moreover $\#(\mathcal{L}^j)=d^2$ for $j=0,1,2$.

\begin{proposition}\label{prop:3d2lines}
With the above notation  $\Phi(\Fd)=\L^0\,\dot\cup \,\L^1\,\dot\cup\, \L^2$.  Thus   $\#(\Phi(\Fd))=3d^2$.
\end{proposition}
\begin{proof} Let us consider the line  $L=Z(x,y)$ in $\pp^3$. Note that we can stratified the lines in $\Fd$ studying their intersection with the line $L$, i.e.
$$
\Phi(\Fd)= \Big\{\ell\in \Phi(\Fd) \ | \  \ell \cap L \neq \emptyset\Big\}\,\,\dot\cup\,\,\Big\{\ell\in \Phi(\Fd) \ | \  \ell \cap L =\emptyset\Big\}.
$$ Let $\ell$ be a line in $\Fd$. Have in mind that $\Fd\cap L=\Big\{[0:0:1:\eta^j]\Big\}_{j=0}^{d-1}$ where $\eta$ is a primitive $d$th root of the unity. Therefore, according to $\ell\cap L \neq \emptyset$ or $\ell\cap L = \emptyset$ we have, respectively:
\begin{itemize}
\item $\ell$ is determined by the points ${\rm p}=[0:0:1:\eta^k]\in L$ for exactly one value of $k\in\{0,\ldots,d-1\}$ (since $L\not\in\Phi(\Fd)$) and ${\rm q}=[\alpha :\beta:0:\gamma]$ with $\alpha, \beta,\gamma\in\cc$ not all zero. Thus
\begin{equation*}
\begin{array}{lcl}
 \ell\subset\Fd & \Longleftrightarrow & \alpha^dv^d-\beta^dv^d-u^d+(\eta^ku+\gamma v)^d =0\quad\forall\,[u:v]\in\pp^1.\\
  & \Longleftrightarrow &
  \left\{
\begin{array}{rl}
      \alpha^d - \beta^d+\gamma^d=0 &\\
    \eta^{k(d-j)}\gamma^j=0 & \hbox{ for }j=1,\ldots,d-1.
\end{array}
  \right.\\
  & \Longleftrightarrow &  \alpha^d-\beta^d=0\,\, (\alpha\beta\neq 0)\,\,\hbox{ and }\,\, \gamma=0.\\
  & \Longleftrightarrow & \ell=Z(w-\eta^kz,y-\eta^ix)=L^0_{k,i}\in\L^0\,\,\hbox{ with }\,\, \alpha^{-1}\beta=\eta^i.
 \end{array}
\end{equation*}
\end{itemize}
Therefore, $\L^0=\Big\{\ell\in \Phi(\Fd) \ | \  \ell \cap L \neq \emptyset\Big\}.$
\begin{itemize}
\item If $\ell\cap L=\emptyset$, then we can assume that $\ell$ is defined by 
$$
x-\alpha z-\beta w\,\,\hbox{ and }\,\, y-\gamma z-\delta w \,\,\hbox{ with }\,\, \alpha\delta -\beta\gamma\neq 0.
$$ Thus,
\begin{equation}
\begin{array}{lcl}
 \ell\subset\Fd & \Longleftrightarrow & 
 
  (\alpha z+\beta w)^d - (\gamma z+ \delta w)^d+z^d-w^d=0.\\
  & \Longleftrightarrow &
  \left\{
\begin{array}{rl}\label{eq:lLvazio}
      \alpha^d -\gamma^d+1=0 &\\
      \beta^d -\delta^d-1=0 & \\
    \alpha^{d-j}\beta^j - \gamma^{d-j}\delta^j=0& \hbox{ for }j=1,\ldots,d-1.
\end{array}
  \right.
 \end{array}
\end{equation}
From (\ref{eq:lLvazio}) for $j=1,2$ (as $d\geq 3$), we get $\alpha^{d-1}\beta=\gamma^{d-1}\delta$ and $\alpha^{d-2}\beta^2=\gamma^{d-2}\delta^2$, which implies that $\gamma^{d-2}\delta(\gamma\beta-\alpha\delta)=0$. Therefore, $\gamma\delta=0$. In fact, we have 
\begin{equation*}
\left\{
\begin{array}{l}
\gamma =0 \Longrightarrow \beta =0 \Longrightarrow \ell\in \L^2.\\
\delta =0 \Longrightarrow \alpha =0 \Longrightarrow \ell\in \L^1.
\end{array}
\right.
\end{equation*}
\end{itemize}
Finally, note  that $[\eta^{k+i}:\eta^i:1:1]\in L^1_{k,i}\-- L^2_{t,j}$ for any $t,j$. Thus,  $\L^1\cap\L^2=\emptyset$.  
\end{proof}

\subsection*{Studying the intersections between the lines on $\Fd$}

 In what follows we use the notation $a\equiv_d b$ instead of $a\equiv b\, ({\rm mod}\,d)$ to indicate that $a$ is congruent to $b$ modulo $d$.

\begin{proposition}\label{PropIntersecoesLs} With the notation as in $(\ref{retasL0L1L2})$. For any  $a,b,i,j,k,t\in \{0,1,\ldots,d-1\}$ holds
\begin{enumerate}\setlength\itemsep{0.1cm}
\item[$(\rm a)$] $ L_{a,b}^0\cap L_{k,i}^s\not=\emptyset \iff \left\{
\begin{array}{ll}
 a=k \,\,\hbox{ or }\,\, b=i    & \hbox{ if }\,\, s=0,\\
  b-a\equiv_d k    &  \hbox{ if }\,\, s=1,\\
  b+a\equiv_d k    &  \hbox{ if }\,\, s=2.
\end{array}
\right.
$
\item[$(\rm b)$] $ L_{k,i}^s\cap L_{t,j}^{s_1}\not=\emptyset \iff \left\{
\begin{array}{ll}
k+i\equiv_d t+j \,\,\hbox{ or }\,\, i=j
   & \hbox{ if }\,\, s=s_1\in\{1,2\},\\
  v^2\eta^{t+2j} = \eta^{k+2i}    &  \hbox{ if }\,\, s=1,\,s_1=2.
\end{array}
\right.
$
\item[$(\rm c)$]
If $d$ is odd, then we can choose $v=-1$ and it follows that 
\begin{center}
$ L_{k,i}^1\cap L_{t,j}^2\not=\emptyset \iff k+2i  \equiv_d t+2j$. 
\end{center}
\item[$(\rm d)$]
If $d$ is even, then  
\begin{center}
$ L_{k,i}^1\cap L_{k,j}^2=\emptyset$ for all  $i,j$. 
\end{center}
\end{enumerate}
\end{proposition}
\begin{proof} The statements (a) and (b) are straightforward  verification (from the definitions of  the lines $L^s_{k,i}$ in (\ref{retasL0L1L2})), and (c) follows from (b). 

Now, let us consider $d\geq 4$ even and suppose that $ L_{k,i}^1\cap L_{k,j}^2\neq\emptyset$ for some $i,j$. Thus, follows from (b) that $v^2\eta^{2j} = \eta^{2i}$, which implies that $v\eta^{j} = \pm\eta^{i}$. Here, if we  compute the $d$-th power of $v\eta^{j} = \pm\eta^{i}$, we lead to an absurd result.
\end{proof}


The results of Proposition~\ref{PropIntersecoesLs} are not novel. In fact, these intersection numbers were previously computed in (\cite{Schutt}, eq. (6) on p. 1944). We became aware of this only after completing our own calculations.

\section{Characterizing sets of skew lines on $\Fd$}

 Let ${\frak s}(X)$ be the \emph{maximal number of skew lines} in $X\subseteq \Fd$. The relations in the above proposition allow us to show that.

\begin{corollary} \label{Co:r(L)}
 ${\frak s}(\mathcal{L}^s)=d,$ for $s=0,1,2$. In particular, ${\frak s}(\Fd)\leq 3d$.
\end{corollary}
\begin{proof} From (a) in Proposition \ref{PropIntersecoesLs} we have that 
$L^0_{a,b}\cap L_{k,i}^0=\emptyset$ iff 
 $a\neq k$ and $b\neq i $. Thus, any subset $\rc\subset\L^0$ of pairwise disjoint lines is constituted by lines $L^0_{a,b}$, of which the indices $a$ are all distinct. Hence, ${\frak s}(\L^0)\leq d$ (since $a\in \{0,\ldots,d-1\}$). On the other hand $\{L^0_{a,a}\}_{a=0}^{d-1}$ is a family of $d$ skew lines in $\L^0$. Therefore, ${\frak s}(\mathcal{L}^0)=d$. 

 One more time, from (b) in Proposition \ref{PropIntersecoesLs} we have that 
$L^s_{k,i}\cap L_{t,j}^s=\emptyset$ iff 
 $i\neq j$ and  $k+i\not\equiv_d t+j$. Again, the condition $i\neq j$ (with $i,j\in \{0,1,...,d-1\})$ implies  that ${\frak s}(\L^s)\leq d$ for $s=1,2$. However, $\{L^s_{k,i}\}_{i=0}^{d-1}$ is constituted by $d$ skew lines in $\L^s$. Therefore, ${\frak s}(\mathcal{L}^s)=d$ for $s=1,2$.
 Finally, note that ${\frak s}(\Fd)\leq {\frak s}(\mathcal{L}^0)+{\frak s}(\mathcal{L}^1)+{\frak s}(\mathcal{L}^2)=3d$.\end{proof}

 From Corollary \ref{Co:r(L)} we have the upper bound $3d$ for ${\frak s}(\Fd)$. So we are invited to look for maximal subsets of skew lines in $\Fd$. In this regard, an important tool is the next proposition, which will establish some kind of  {\it sudoku}'s rule for our game\footnote{ The game is: given $d\geq 4$ find the maximal number of pairwise lines on Fermat surface $\Fd$.}. In fact, the lower bound $2d$ for ${\frak s}(\Fd)$ will be established in Corollary \ref{Co:lbound2d}. From this point onward, we start playing (pay attention to the rules!).
\begin{proposition}\label{PropPsiPhis}
Let $R_d=\{0,1,\ldots,d-1\}$ and  $r_d:\mathbb{Z}\longrightarrow R_d$ the remainder\footnote{If $a\in\mathbb{Z}$, then $r_d(a)=r$ where $r\in R_d$ and $a\equiv_d r$.} function by $d$. Consider the functions
$$
\begin{array}{ccccccccc}
    \psi_d: & R_d\times R_d & \longrightarrow & R_d &  \,\,\hbox{ and }\,\, &\varphi_{d,{\pm}}: & R_d\times R_d & \longrightarrow & R_d \\
    & (k,i) & \longmapsto & r_d(k+2i) & &
    & (k,i) & \longmapsto & r_d(i\pm k).
\end{array}
$$
For $u \in R_d$, $s\in\{0,1,2\}$ define
\begin{align*}
    {\rm D}_u^s&:=\big\{\mathcal{L}_{k,i}^s\in \mathcal{L}^s\ \mid \ (k,i)\in \psi_d^{-1}(u)  \big\}, \hbox{ for }s=1,2; \\
    {\rm D}_{u,\pm }^0&:=\big\{\mathcal{L}_{k,i}^s\in \mathcal{L}^0\ \mid \ (k,i)\in \varphi_{d,\pm}^{-1}(u)  \big\}.
\end{align*} 
It is verified that 
\begin{enumerate}
\item[$(\rm a)$] the restriction of $\psi_d$ and $\varphi_{d,\pm}$ to $R_d\times \{i\}$ is a bijection for all $i$;
\item[$(\rm b)$]  $\restr{\varphi_{d,\pm}}{\{k\}\times R_d } : \{k\}\times R_d \to R_d$   is a bijection for all $k$;
    \item[$(\rm c)$] $\psi_d \mid_{\{k\}\times R_d } : \{k\}\times R_d \to R_d$ is a bijection for all $k$, if $d$ is odd;
    \item[$(\rm d)$] $\#\psi_d^{-1}(u)=d$ and $\#{\varphi_{d,\pm}}^{-1}(u)=d$ for all $u \in R_d$;
     \item[$(\rm e)$]   $ {\rm D}_{u,\pm }^0\subset \L^0$ and ${\rm D}_u^s\subset\L^s$ for $s=1,2,$   are  families of $d$ skew lines.
\end{enumerate}
\end{proposition} 
\begin{proof} It is left to the reader.\end{proof}

 In the next corollary we find the lower bound $2d$ for ${\frak s}(\Fd)$.

\begin{corollary}\label{Co:lbound2d}
 ${\frak s}(\mathcal{L}^0\cup \mathcal{L}^s)=2d$ for $s=1,2$ and ${\frak s}(\mathcal{L}^1\cup \mathcal{L}^2)=2d$. Thus $2d\leq {\frak s}(\Fd)\leq 3d$.    
\end{corollary}

\begin{proof} From Corollary \ref{Co:r(L)} we have that ${\frak s}(\mathcal{L}^0)={\frak s}(\mathcal{L}^1)={\frak s}(\mathcal{L}^2)=d$. Which implies that $${\frak s}(\mathcal{L}^0\cup \mathcal{L}^s)\leq 2d, \textrm{ for } s=1,2 \hspace{1cm} \textrm{ and } \hspace{1cm} {\frak s}(\mathcal{L}^1\cup \mathcal{L}^2)\leq 2d.$$ Thus, it is enough to find a family of $2d$ skew lines on $\mathcal{L}^0\cup \mathcal{L}^s$ for $s=1,2$ and on $\mathcal{L}^1\cup \mathcal{L}^2$, respectively. For the first statement, from  item (a) of Proposition \ref{PropIntersecoesLs}, we conclude that the  following two sets are constituted by $2d$ skew lines 
$$\{L^0_{0,0},L^0_{1,1},\ldots,L^0_{d-1,d-1},L_{1,0}^1,L_{1,1}^1,\ldots,L_{1,d-1}^1\};$$ 
$$\{L^0_{0,0},L^0_{1,d-1},L^0_{2,d-2},\ldots,L^0_{d-1,1},L_{1,0}^2,L_{1,1}^2,\ldots,L_{1,d-1}^2\}.
$$ 
Now, for the second statement, we have that  $\#{\rm D}_{0}^1=d$ and $\#{\rm D}_{1}^2=d$ in accordance with  the item (e) of Proposition \ref{PropPsiPhis}. Moreover, by item (b) of Proposition \ref{PropIntersecoesLs} we have that $L_{k,i}^1\cap L_{m,n}^2=\emptyset$ for any $L^1_{k,i}\in {\rm D}_0^1$ and $L_{m,n}^2\in {\rm D}_1^2$. Therefore, ${\rm D}_0^1\cup {\rm D}_1^2$ is a family of $2d$ skew lines in $\L^1\cup\L^2$.
\end{proof}

From now on, we will focus on capturing maximal subsets of skew lines in $\Fd$, revisiting the conditions that must be satisfied by such subsets.

\subsubsection{Rewriting conditions for subsets of skew  lines in $\Fd$}

In order to find maximal sets of skew lines in $\Fd$, we started by characterizing those subsets of skew lines in $\L^s$ for each $s=0,1,2$ in terms of $\psi_d$ and $\varphi_{d,\pm}$ (cf. Proposition \ref{PropPsiPhis}), 
when it comes to the case.



Once again, from the Proposition~\ref{PropPsiPhis} we obtain the following two corollaries.

\begin{corollary}\label{cor:Ciskew} Let $\rc\subset\Phi(\Fd)$ and define $\rc^s:=\rc\cap\L^s$ for $s\in\{0,1,2\}$. 
\begin{itemize}
\item[$({\rm a})$] $\rc^0$ is constituted by skew lines $\iff$
$\left\{\textrm{\begin{tabular}{l}
     $\rc^0=\{L^0_{a_1,b_1},\ldots, L^0_{a_m,b_m}\}$ with $\#\rc^0=m$,\\ $0\leq a_1<\cdots <a_m\leq d-1$ and
     there is\\ a permutation $\sigma$ of $R_d$ such that $\sigma (a_i)=b_i$.
\end{tabular}}\right.$

\item[$({\rm b})$] 
$\rc^1$ is constituted by skew lines $\iff$ $\left\{
\textrm{\begin{tabular}{l}
     $\rc^1=\{L^1_{a_1,b_1},\ldots, L^1_{a_m,b_m}\}$ with $\#\rc^1=m$,\\
     $0\leq b_1<\cdots <b_m\leq d-1$ and $\varphi_{d,+}$ \\restricted   to $\{(a_i,b_i)\}_{i=1}^m$ is injective.
\end{tabular}}\right.$
\item[$({\rm c})$] $\rc^2$ is constituted by skew lines $\iff$ $\left\{\textrm{\begin{tabular}{l}
     $\rc^2=\{L^2_{a_1,b_1},\ldots, L^2_{a_m,b_m}\}$ with $\#\rc^2=m$,\\
     $0\leq b_1<\cdots <b_m\leq d-1$ and $\varphi_{d,+}$ \\restricted   to $\{(a_i,b_i)\}_{i=1}^m$ is injective.
\end{tabular}}\right.$
\end{itemize}
\end{corollary}
\begin{remark} Note that 
$
\L^s_k=\{L^s_{k,i}\in\L^s\mid i\in R_d\}$ for $ k\in R_d $ and $ s\in\{0,1,2\}$ 
is constituted by $d$  skew lines if $s\in\{1,2\}$ (according to $({\rm b})$ in Proposition \ref{PropIntersecoesLs}). Moreover,  
$$\L^s=\L^s_0\,\,\dot\cup\,\cdots \,\dot\cup\,\L^s_{d-1}\quad\hbox{ for any }\, s\in\{ 0,1,2\}.$$
Now, we will concentrate our attention on the description of those  subsets $\rc^s$ of $\L^s$ consisting of skew lines such that $\rc^s\cup\rc^{s_1}$ is also formed by skew lines (for $0\leq s<s_1\leq 2$).
\end{remark}

In what follows, for any subset ${\rm X}\subseteq\Phi(\Fd)$ we may identify the line $\L^s_{k,i}\in {\rm X}$ with the pair $(k,i)$ (which will be clear from the context). Having this in mind we will consider  $\psi_d({\rm X})$ and $\varphi_{d,\pm}({\rm X})$.

\begin{corollary}\label{Cor:CiCjskew} With the above notation. Assume that $\rc^0, \rc^1$ and $\rc^2$ consist of skew lines. Then we have
\begin{itemize}
\item[$({\rm a})$] $\rc^0\,\cup\,\rc^1$ is constituted by skew lines $\iff$ 
 $\rc^1\cap\L^1_k=\emptyset$ for every $k\in \varphi_{d,-}(\rc^0)$.
\item[$({\rm b})$]
$\rc^0\,\cup\,\rc^2$ is constituted by skew lines $\iff$ 
 $\rc^2\cap\L^2_k=\emptyset$ for every $k\in \varphi_{d,+}(\rc^0)$.
\item[$({\rm c})$] For $d$ odd holds $\rc^1\,\cup\,\rc^2$ is constituted by skew lines $\iff$ 
     $\psi_d(\rc^1)\,\cap\, \psi_d(\rc^2)=\emptyset$.
\end{itemize}
\end{corollary}

\begin{remarks}\label{Obs:phi=rd} Assume $d$ odd. If $\rc^s\subset\L^s$ consists of skew lines for $s=0,1,2$, then  Corollary \ref{Cor:CiCjskew} allows to conclude that
\begin{itemize}
  \item  If  $\rc^1=\L^1_k$ $($resp. $\rc^2=\L^2_k)$ for some $k\in R_d$, then $\rc^2=\emptyset$  $($resp. $\rc^1=\emptyset)$. 
  \item If $ \varphi_{d,+}(\rc^0)=R_d$ $($resp. $\varphi_{d,-}(\rc^0)=R_d$), then $\rc^2=\emptyset$ $($resp. $\rc^1=\emptyset )$\footnote{It is also true for $d$ even.}. 
  \item $\#\psi_d(\rc^1)+\#\psi_d(\rc^2)\leq d.$ In particular, if $\psi_{d}(\rc^1)=R_d$, then $\rc^2=\emptyset$ and vice versa.
\end{itemize}
\end{remarks}

Let's see an example of a family of 13 skew lines in
${\rm F}_5$.

\begin{example}\label{exampled=5} Note that $\rc^0=\{ L^0_{0,4}, L^0_{2,0}, L^0_{3,1}, L^0_{4,3}\}$ consists of four skew lines (cf. (a) in Corollary \ref{cor:Ciskew}). Furthermore, in the rows of the next table we register  the values of $\varphi_{5,\pm}(\rc^0)$, respectively.
\begin{center}
\begin{tabular}{|c||c|c|c|c|} \hline
         $\rc^0$    & $L^0_{0,4}$ & $L^0_{2,0}$  & $L^0_{3,1}$ & $L^0_{4,3}$    \\ \hline \hline
          $\varphi_{5,-}$   & 4 & 3 & 3 & 4 \\ 
          \hline
          $\varphi_{5,+}$   & 4 & 2 & 4 & 2 \\ \hline
        \end{tabular}
\end{center}

Now, having in mind  Corollary \ref{Cor:CiCjskew} for the choice of $\rc^s\subset\L^s$ such that $\rc^0\cup\rc^s$ is constituted by skew lines for $s=1,2$, it is necessary that 
$$
\rc^1\cap\L^2_k=\emptyset\quad\forall\, k\in\{3,4\}\quad\hbox{and}\quad \rc^2\cap\L^2_k=\emptyset\quad\forall\, k\in\{2,4\}.
$$ So,   $\rc^1=\{ L^1_{0,1},L^1_{0,4},L^1_{2,0},L^1_{2,3}\}$ and $\rc^2=\{ L^2_{0,0},L^2_{1,2},L^2_{3,1},L^2_{3,3},L^2_{3,4}\}$ are admissible choices. As well as according to the information on the rows in   the following two  tables.

\vspace{0.1cm}

\begin{center}
\begin{minipage}{0.4\textwidth}
\begin{tabular}{|c||c|c|c|c|} \hline
         $\rc^1$    & $L^1_{0,1}$  & $L^1_{0,4}$ & $L^1_{2,0}$ & $L^1_{2,3}$  \\ \hline \hline
          $\varphi_{5,+}$   & 1 & 4 & 2 & 0 \\
          \hline
          $\psi_5$   & 2 & 3 & 2 & 3 \\ \hline
        \end{tabular}	
\end{minipage}
\begin{minipage}{0.4\textwidth}
\begin{tabular}{|c||c|c|c|c|c|} \hline
         $\rc^2$    & $L^2_{0,0}$  & $L^2_{1,2}$ & $L^2_{3,1}$ & $L^2_{3,3}$  & $L^2_{3,4}$  \\ \hline \hline
          $\varphi_{5,+}$   & 0 & 3 & 4 & 1 & 2 \\ 
          \hline
          $\psi_5$   & 0 & 0 & 0 & 4 & 1\\ \hline
        \end{tabular}	
\end{minipage}
\end{center}
\vspace{0.1cm}

We have that $\rc^s$ consists of skew lines for $s=1,2$ (cf. (b) and (c) in Corollary \ref{cor:Ciskew}) and $\psi_5(\rc^1)\cap \psi_5(\rc^2)=\emptyset$, which implies that $\rc^1\cup \rc^2$ also is formed by skew lines (cf. (c) in Corollary \ref{Cor:CiCjskew}). Therefore, $\rc:=\rc^0\cup\rc^1\cup\rc^2$ consists of 13 skew lines in ${\rm F}_5$. In fact, in Theorem~\ref{Teod=5}, we will prove that ${\frak s}({\rm F}_5)=13 $.
\end{example}
        
\begin{remarks}\label{ObsMatrix} 
Note that the correspondence between lines in $\L^s$ and pairs in $R_d\times R_d$ for each $s\in\{0,1,2\}$ allows us to associate to the $d^2$ lines  in $\L^s$ (for each $s\in\{0,1,2\}$) the following $d\times d$ square matrix:
\begin{equation}\label{matrixd2}
\left(
\begin{array}{cccc}
    (0,0) & (1,0) & \cdots & (d-1,0) \\
    (0,1) & (1,1) & \cdots & (d-1,1) \\
     \vdots & \vdots & \cdots & \vdots \\
     (0,d-1) & (1,d-1) & \cdots & (d-1,d-1) 
\end{array}
\right).
\end{equation} 

Now, let us investigate the families of lines identified by the entries in the rows, columns, diagonals, and anti-diagonals of the matrix mentioned above. But first of all, it is important to make  clear that:

For each $r\in R_d$
\begin{itemize}
\item $\varphi_{d,-}^{-1}(r)$ will be named a   {\sf diagonal with remainder $r$} of the matrix in  (\ref{matrixd2}).
\item $\varphi_{d,+}^{-1}(r)$ will be named a  {\sf anti-diagonal with remainder $r$} of the matrix in  (\ref{matrixd2}).
\end{itemize}
 
 So, for example we say that $L^s_{a,b}$ and $L^s_{a_1,b_1}$ are in the same diagonal (resp. anti-diagonal) if $\varphi_{d,-}(a,b)= \varphi_{d,-}(a_1,b_1)$ (resp. $\varphi_{d,+}(a,b)= \varphi_{d,+}(a_1,b_1)$).

\vspace{0.2cm}

Note that
\begin{itemize}
\item[$({\rm i})$] The family $\L^s_k$ 
is labeled by the pairs in the $(k+1)$-th column of  the matrix in (\ref{matrixd2}).
\item[$({\rm ii})$] Any two lines labeled by pairs in the same row of the matrix in (\ref{matrixd2}) meet. 
\item[$({\rm iii})$] Each of such  diagonals and anti-diagonals determines $d$ skew lines in $\L^0$.
\item[$({\rm iv})$] Each diagonal (resp. anti-diagonal) with remainder $r$ meets the {\sf column} in the matrix (\ref{matrixd2}) in exactly one pair (i.e. in exactly one line in $\L^s_k$ for $k\in R_d$).
\item[$({\rm v})$]  Let  $L^s_{a,b},L^s_{a_1,b_1}\in\L^s$ be  disjoint and $d$ odd and $s\in \{0,1,2\}$. If $\varphi_{d,\pm}(a,b)= \varphi_{d,\pm}(a_1,b_1)$, then $\varphi_{d,\mp}(a,b)\neq \varphi_{d,\mp}(a_1,b_1)$. In other words, if $L^s_{a,b},L^s_{a_1,b_1}$ are lines on the same  diagonal,  then they are in distinct anti-diagonal, and vice versa. 
\end{itemize}
 \end{remarks}

\section{Computing ${\frak s}(\Fd)$ for $d\in\{3,5\}$}\label{Secd35}

Next we will exhibit the only two Fermat surfaces $\Fd$ satisfying ${\frak s}(\Fd)<3d$.

\subsection{Showing that ${\frak s}({\rm F}_3)=6$}


\begin{proposition}\label{Prop:d=3} Let $\rc$ be a set of skew lines on Fermat cubic ${\rm F}_3$ and consider $\rc ^s=\rc\cap \mathcal{L}^s$ for each $s=0,1,2$. If $\#\rc^s=3$ then there exists $k\in \{0,1,2\}\setminus \{s\}$ such that $\rc^k=\emptyset$. \end{proposition}

\begin{proof} Next, we will subdivide in the following three cases:

$\bullet$ Assume $\#\rc ^0=3$ and let $\rc ^0=\left\{L^0_{0,b_0},L^0_{1,b_1},L^0_{2,b_2}\right\}$ 
 with $b_j\in \{0,1,2\}$ and 
 \begin{equation}\label{Eq:congbi}
     b_0\not = b_1,\,\, b_0\not = b_2,\,\, b_1 \not = b_2.
 \end{equation}
We claim that $\#\varphi_{3,+}(\rc^0)=1$ or $\#\varphi_{3,-}(\rc^0)=1$. 
Note that $r_3(\{b_i,b_i+1,b_i+2\})=R_3$, for any $i.$ Thus $b_0\equiv_3 b_1+j$ for some $j=0,1,2.$ 
 In fact, $b_0 \not \equiv_3b_1$, so we have
 $$
  \underbrace{b_0 \equiv_3 b_1 +1}_{(i)}\quad\hbox{ or }\quad \underbrace{b_0  \equiv_3b_1+2.}_{(ii)}  
 $$

For $(i)$, have in mind that $b_1+1\equiv_3 b_2+j$ for some $j=0,1,2.$ In fact, by Equation~\eqref{Eq:congbi} we have
\begin{align*}
    b_1+1\not \equiv_3 b_2 +1 \quad \hbox{ and } \quad      b_1+1 \not \equiv_3b_2. 
 \end{align*}   
Thus $b_0\equiv_3 b_1+1\equiv_3 b_2+2$ and consequently  $\varphi_{3,+}(\rc^0)=\{b_0\}$, so $\#\varphi_{3,+}(\rc^0)=1$.

\vspace{0.2cm }

For $(ii)$ we used  that $b_1+2\equiv_3 b_2+j$ for some $j=0,1,2.$ 
 However, by Equation~\eqref{Eq:congbi} we have
\begin{align*}
    b_1+2\not \equiv_3 b_2  \quad \hbox{ and } \quad 
     b_1+2 \not \equiv_3b_2 +2  \hbox{.}
 \end{align*}   
Thus 
$b_0\equiv_3 b_1+2\equiv_3 b_2+1$, that is, $b_0\equiv_3 b_1-1\equiv_3 b_2-2$ and consequently  $\varphi_{3,-}(\rc^0)=\{b_0\}$, so $\#\varphi_{3,-}(\rc^0)=1$.

Finally, if $\#\varphi_{3,+}(\rc^0)=1$ then $\#\varphi_{3,-}(\rc^0)=3$ (cf. item (v) Remark \ref{ObsMatrix}). Which implies that    $\rc^1=\emptyset$ (cf. Remark \ref{Obs:phi=rd}). Analogously, if $\#\varphi_{3,-}(\rc^0)=1$ then $\rc^2=\emptyset$.

$\bullet$ Assume $\#\rc^1=3$ and let $\rc^1=\left\{L^1_{a_0,0},L^1_{a_1,1},L^1_{a_2,2} \right\}$ with $a_i\in \{0,1,2\}$ and 
\begin{equation}\label{Eq:congai}
a_0\not \equiv_3 a_1 +1,\,\, a_0\not \equiv_3 a_2 +2,\,\, a_1+1\not \equiv_3 a_2 +2. 
\end{equation} 
We will analyze the following two possibilities: $a_0\equiv_3a_1+2$ or  $a_0\not\equiv_3a_1+2.$

\vspace{0.1cm}

\noindent\fbox{$a_0\equiv_3a_1+2$} One more time have in mind that $a_1+2\equiv_3 a_2+j$ for some $j=0,1,2.$ 
 In fact, follows from Equation~\eqref{Eq:congai}\footnote{Note that $a_1+2 \equiv_3a_2 \Longrightarrow a_1+1 \equiv_3a_2+2$. As well as, $a_1+2 \equiv_3 a_2 +2\Longrightarrow a_0 \equiv_3 a_2 +2$.} that
\begin{align*}
    a_1+2 \not \equiv_3a_2 \quad \hbox{ and } \quad     a_1+2\not \equiv_3 a_2 +2.
 \end{align*}
Thus $a_0\equiv_3a_1+2\equiv_3 a_2+1$. This implies that $\#\psi_3(\rc^1)=1$ and therefore $\rc^0=\emptyset$.\footnote{$\#\psi_3(\rc^1)=1 \implies a_0\equiv_3 a_1+2\equiv_3 a_2+1 \implies \{a_0,a_1,a_2\}=R_3 \implies \rc^0=\emptyset$.}




\hspace{-0.4cm}\fbox{$a_0\not\equiv_3a_1+2$} In this case we will show that $\#\psi_3(\rc^1)=3$ (i.e., $r_3(\{a_0,a_1+2,a_2+1\})=R_3$). 

Since $a_0\not \equiv_3 a_1+1$ (cf. \eqref{Eq:congai}) then necessarily $a_0\equiv_3 a_1$. On the other hand, note that 
$$
a_1+2\equiv_3 a_2+1\Longrightarrow a_0+2\equiv_3 a_2+1 \Longrightarrow a_0\equiv_3 a_2+2
$$\vspace{-0.3cm} and
$$a_0\equiv_3 a_2+1 \implies a_1+1\equiv_2 a_2+2$$
which are both absurd (cf. \eqref{Eq:congai}). 
Therefore $a_0\not\equiv_3a_1+2, a_0\not\equiv_3 a_2+1, a_1+2\not \equiv a_2+1$ and this implies that $\psi_3(\rc^1)=R_3$. Furthermore $\rc^2=\emptyset$ (cf. Remarks~\ref{Obs:phi=rd}). The case where $\#\rc^2=3$ we left as an exercise for the reader. \end{proof}

\begin{corollary} ${\frak s}({\rm F}_3)=6$.
\end{corollary}

\begin{proof} 
Let $\rc\subset {\rm F}_3$ be a set of skew lines such that $\#\rc>6$. Then $\rc \cap \mathcal{L}^i=\rc^i\not = \emptyset,$ for each $i=0,1,2$ (cf. Corollary~\ref{Co:r(L)}). By Proposition~\ref{Prop:d=3} we may conclude that $\#\rc^i\leq 2$ for each $i=0,1,2$ which is an absurd. This implies that ${\frak s}({\rm F}_3)\leq6$. Now use Corollary~\ref{Co:lbound2d}. \end{proof}

\subsection{Showing that ${\frak s}({\rm F}_5)=13$}

\begin{lemma}\label{Lemma01d5} Let $\rc^0\subset\mathcal{L}^0$ be a set of skew lines in ${\rm F}_5$. If $\#\rc^0=5$ then $\#\varphi_{5,+}(\rc^0)\geq 3$ or $\#\varphi_{5,-}(\rc^0)\geq 3$.\end{lemma}

\begin{proof} Assume that  $\rc^0=\{L^0_{a_0,b_0},\ldots,L^0_{a_4,b_4}\}$. If $\#\varphi_{5,+}(\rc^0)\leq 2$ then, without lost of generality, we may assume that $a_0+b_0\equiv_5 a_1+b_1\equiv_5a_2+b_2.$ This implies that $b_0-a_0\not \equiv_5 b_1-a_1$, $b_0-a_0\not \equiv_5 b_2-a_2$ and $b_1-a_1\not \equiv_5 b_2-a_2$ (cf. (v) in Remarks \ref{ObsMatrix}). Therefore $\#\varphi_{5,-}(\rc^0)\geq 3$ as we desired.
\end{proof}

\begin{lemma}\label{lemma02d5} Let $\rc^s=\{L^s_{a_0,b_0},\ldots,L^s_{a_4,b_4}\}\subset \mathcal{L}^s$    be a set of skew lines in ${\rm F}_5$ such that $\#\rc^s=5$ for $s=1,2$. If $\#\{a_0,\ldots,a_4\}\leq 3$ then $\#\psi_5(\rc^s)\geq 3$.\end{lemma}

\begin{proof} We will divide the proof in three cases according to the $\#\{a_0,\ldots,a_4\}$. The first case is $\#\{a_0,\ldots,a_4\}=1$. In this case it follows that $\#\psi_5(\rc^s)=5$ since $\rc^s=\L^s_k$ for some $k\in R_5$. The second one is when $\#\{a_0,\ldots,a_4\}=2$ and in this case at least three are equal, so we may assume that $a_0=a_1=a_2$. This implies that $\#\psi_5(\rc^s)\geq 3$.\footnote{Since, $a_0+2b_i\equiv_5 a_0+2b_j\Longleftrightarrow b_i\equiv_5 b_j$ for $i\neq j$, and $i,j\in\{0,1,2\}$} The last one occurs when $\#\{a_0,\ldots,a_4\}=3$. In this case we have two possibilities (reordering indexes if necessary):
\\
[1mm]
(i) $a_0=a_1=a_2$ and $\#\{a_0,a_3,a_4\}=3$, which implies $\#\psi_5(\rc^s)\geq 3$.
\\
[1mm]
(ii) $a_0=a_1$, $a_2=a_3$ and $\#\{a_0,a_2,a_4\}=3$. In this case,
$$
a_0+2b_0\not\equiv_5 a_1+2b_1\,\hbox{and }\, a_2+2b_2\not\equiv_5 a_3+2b_3,
$$ which implies that $\#\psi_5(\rc^s)\geq 2$. Let us suppose by absurd that $\#\psi_5(\rc^s)=2$. So we may assume that $$a_0+2b_0\equiv_5 a_2+2b_2 \equiv_5 a_4+2b_4\,\textrm{ and }\, a_1+2b_1\equiv_5 a_3+2b_3.$$ Now, note that
\begin{align*}
		\sum_{i=0}^4(a_i+b_i)\equiv_5 &\ \sum_{i=0}^4a_i\equiv_5 \sum_{i=0}^4(a_i+2b_i) \equiv_5 3(a_0+2b_0)+ 2(a_1+2b_1) \\ \equiv_5  &\ (a_0+b_0)+2a_0+(a_1+b_1)+a_1+3b_1 
	\end{align*}
 which implies that $(a_2+b_2)+(a_3+b_3)+(a_4+b_4) \equiv_5 2a_0+a_1+3b_1$. Having in mind that $r_5(\{a_i+b_i\}_{i=1}^5)=R_5$ (since $\#\rc^s=5$), we  have that
 $$
 (a_2+b_2)+(a_3+b_3)+(a_4+b_4) \equiv_5 4(a_0+b_0)+4(a_1+b_1). 
 $$
 Thus,\begin{align*}
	  4(a_0+b_0)+4(a_1+b_1) \equiv_5 2a_0+a_1+3b_1  \implies & 3a_1+b_1+2a_0+4b_0\equiv_5 0 \\ \implies & 2(a_0+2b_0)\equiv_5 2(a_1+2b_1)\\ \implies & a_0+2b_0\equiv_5a_1+2b_1  
\end{align*} and this is an absurd.  Therefore $\#\psi_5(\rc^s)\geq 3$ for $s=1,2$.\end{proof}

\begin{lemma}\label{lemma03d5} Let $\rc$ be a set of skew lines in ${\rm F}_5$, $\rc^i=\rc\cap \mathcal{L}^i$ for $i=0,1,2$ such that $\#\rc^0\geq 4$. If $\rc^s=\{ L^s_{a_0,b_0},\ldots,L^s_{a_4,b_4}\}$ with $\#\rc^s=5$ and $\#\rc^r = 4$ where $\{s,r\}=\{1,2\}$ then $\#\{a_0,\ldots,a_4\}\geq 3$.
\end{lemma}

\begin{proof} Note that, if $\#\{a_0,\ldots,a_4\}=1$, then $\rc^s=\L^s_k$ for some $k\in R_5$. And this implies that $\rc^r=\emptyset$ which is an absurd (cf. Remarks \ref{Obs:phi=rd}). Now, if $\#\{a_0,\ldots,a_4\}=2$, then we have two possibilities (reordering indexes if necessary): 

(i) $a_0=a_1=a_2=a_3$ and $a_0\not =a_4\quad $  (ii) $a_0=a_1=a_2$ and $a_0\not = a_3=a_4$.

In case (i) it follows that $\#\psi_5(\rc^s)\geq 4$ and $\#\psi_5(\rc^r)\in\{0,1\}$ (since $\psi_5(\rc^s)\cap \psi_5(\rc^r)=\emptyset$). Hence, if $\#\psi_5(\rc^r)=0$ then $\rc^r=\emptyset$, else $\#\psi_5(\rc^r)=1$ which implies $\#\{a_0, \ldots, a_4\}\leq 1$ and this is an absurd.\footnote{In fact, assume that $s=1,r=2$,  $\rc^2=\{ L^2_{a_0',b_0'},\ldots,L^2_{a_3',b_3'}\}$. As $\#\psi_5(C^2)=1$ then $\#\{a_0',\ldots,a_3'\}=4$ (if $a_0'=a_1'$, then $a_0'+2b_0\not \equiv_5 a_1' +2b_1'$ since $b_0'\not \equiv_5 b_1'$, which is an absurd). Therefore, $\#\varphi_{5,+}(\rc^0)=1$ (cf. Corollary~\ref{Cor:CiCjskew}) which implies that $\#\varphi_{5,-}(\rc^0)\geq 4$. Hence,  $\#\{a_0, \ldots, a_4\}\leq 1$.
} 

For (ii) note that  
\begin{align*}
	0\equiv_5 \sum_{i=0}^4(a_i+b_i)\equiv_5 3a_0+2a_3 \implies 3a_0\equiv_5 3a_3 \implies a_0\equiv_5 a_3 \implies a_0=a_3
\end{align*} which is an absurd.\end{proof}

\begin{theorem}\label{Teod=5} ${\frak s}({\rm F}_5)= 13$.\end{theorem}

\begin{proof} Let $\rc$ be a set of skew lines in ${\rm F}_5$. Let us suppose that $\#\rc\geq 14$. In fact, it is enough to analyse the case $\#\rc= 14.$
Define $\rc^i=\rc\cap \mathcal{L}^i$, with $i=0,1,2$. Note that only one of the  possibilities happens:\begin{enumerate}
	\item[$(a)$] $\#\rc^0=4$ and $\#\rc^1=\#\rc^2=5$;
	\item[$(b)$] $\#\rc^0=5$, $\#\rc^s=5$ and $\#\rc^r=4$  for $\{r,s\}=\{1,2\}$. 
\end{enumerate}
For $(a)$ let us consider $\rc^1=\{L^1_{a_0,b_0},\ldots,L^1_{a_4,b_4}\}$ and $\rc^2=\{L^2_{a_0',b_0'},\ldots,L^2_{a_4',b_4'}\}$. Note that $$\#\{a_0,\ldots,a_4\}\leq 3\ \textrm{ or }\ \#\{a_0,\ldots,a_4\}>3.$$ The last inequality can not occur because other way \begin{enumerate}
	\item[$(i)$] if $\#\{a_0,\ldots,a_4\}=5$ then $\#\varphi_{5,-}(C^0)=0 $ and this implies that $\rc^0=\emptyset$;
	\item[$(ii)$] if $\#\{a_0,\ldots,a_4\}=4$, then $\#\varphi_{5,-}(\rc^0)= 1$. Hence, $\#\varphi_{5,+}(C^0)=4$. Therefore, $\#\{a_0',\ldots,a_4'\}=1$, which implies that $\#\psi_5(C^2)=5$. Furthermore, $\#\psi_5(C^1)=0$ which is an absurd.
\end{enumerate}
Therefore, $\#\{a_0,\ldots,a_4\}\leq 3$. Analogously, we may conclude that $\#\{a_0',\ldots,a_4'\}\leq 3$. It follows from  Lemma~\ref{lemma02d5} that $\#\psi_5(\rc^1)\geq 3$ and $\#\psi_5(\rc^2)\geq 3$, which is an absurd by Remark~\ref{Obs:phi=rd}.

For $(b)$, let us assume that $\#\rc^0=\#\rc^1=5$ and $\#\rc^2=4$ (the other case is analogous). Let us consider $\rc^0=\{L^0_{a_0,b_0},\ldots,L^0_{a_4,b_4}\}$,  $\rc^1=\{L^1_{a_0',b_0'},\ldots,L^1_{a_4',b_4'}\}$  and $\rc^2=\{L^2_{a_0'',b_0''},\ldots,L^2_{a_3'',b_3''}\}$. Using arguments analogous to cases $(i)$ and $(ii)$ we may conclude\footnote{If $\#\{a_0',...,a_4'\}=5$ then $\rc^0=\emptyset$. If $\#\{a_0',...,a_4'\}=4$ then $\#\varphi_{5,-}(\rc^0)= 1$. Hence $\#\varphi_{5,+}(\rc^0)=5$ and this implies that $\rc^2=\emptyset$, which is an absurd.} that $\#\{a_0',\ldots,a_4'\}\leq 3$. Now, it follows from  Lemma~\ref{lemma03d5}  that $\#\{a_0',\ldots,a_4'\}=3$. So $\#\big(\varphi_{5,-}(\rc^0)\big)\leq 2$. Now, we will analyze all three possibilities:\begin{enumerate}
	\item[$(iii)$] if $\#\varphi_{5,-}(\rc^0)=0$, then $\rc^0=\emptyset$;
	\item[$(iv)$] if $\#\varphi_{5,-}(\rc^0)=1$, then $\#\varphi_{5,+}(\rc^0)=5$ and this implies that $\rc^2=\emptyset$;
	\item[$(v)$] if $\#\varphi_{5,-}(\rc^0)=2$, then we may assume that \begin{equation}\label{teoeq1d5}
		b_0-a_0\equiv_5b_1-a_1\equiv_5b_2-a_2\equiv_5 b_3-a_3 \quad \textrm{ and }\quad b_0-a_0 \not \equiv_5 b_4-a_4
	\end{equation}$$ \textrm{ or } $$\begin{equation}\label{teoeq2d5}
		b_0-a_0\equiv_5b_1-a_1\equiv_5b_2-a_2 \quad \textrm{ and }\quad b_0-a_0 \not \equiv_5 b_4-a_4\equiv_5 b_3-a_3
	\end{equation} 
\end{enumerate} By Equation~\eqref{teoeq1d5}, we may conclude that $\#\big(\varphi_{5,+}(\rc^0)\big)\geq 4$. So, $\#\{a_0'',\ldots,a_3''\}=1$. This implies that $\#\psi_5(\rc^2)=4$ and consequently $\#\psi_5(\rc^1)= 1$ by Remark~\ref{Obs:phi=rd}. And this is an absurd by Lemma~\ref{lemma02d5}. Finally, by Equation~\eqref{teoeq2d5} it follows that \begin{align*}
	\sum_{i=0}^4(b_i-a_i)\equiv_50 \implies & 3(b_0-a_0)+2(b_4-a_4)\equiv_5 0 \\ \implies &  3(b_0-a_0)\equiv_5 3(b_4-a_4) \\ \implies &  b_0-a_0\equiv_5 b_4-a_4
\end{align*}
which  is an absurd. Therefore, $\#\rc\leq 13$ for any set $\rc$ of skew lines in $F_5$. On the other hand, the  Example~\ref{exampled=5} shows an example with $13$ skew lines, thus ${\frak s}({\rm F}_5)= 13$.
\end{proof}

\section{Addressing the case $d\geq 4$ and $d\not = 5$}\label{Secdgeq4}
From Corollary \ref{Co:lbound2d}, we have that $2d\leq {\frak s}(\Fd)\leq 3d$ for any $d\geq 3$. For $d\geq 4$ even, we have that ${\frak s}(\Fd)= 3d$, as we prove in the next proposition. However, for $d\geq 7$ odd, we will devide our study into two cases: $d\equiv_4 1$ and $d\equiv_4 3$ being $d\geq 7$.

\begin{proposition}
\label{Prop:dpar}
 Let $d\geq 4$ even. If $\rc^0=\Big\{L^0_{a,a}\Big\}_{a=0}^{d-1}$ and   $\rc^s=\Big\{L^s_{1,i}\Big\}_{i=0}^{d-1}$ for $s=1,2$, then $\rc^0\cup\rc^1\cup\rc^2$ consists of $3d$ skew lines in $\Fd$. 
\end{proposition}
\begin{proof} It follows from Corollary \ref{cor:Ciskew} that $\rc^s$ consists of $d$ skew lines for each $s\in\{0,1,2\}$.\footnote{Note that $\varphi_{d,+}(\rc^1)=R_d=\varphi_{d,+}(\rc^2)$ and $\#\rc^s=d$ for all $s$.} On the other hand, $\varphi_{d,-}(\rc^0)=\{0\}$ and $\varphi_{d,+}(\rc^0)=\{0,2,\ldots, 2d-2\}$, which implies that $\rc^0\cup\rc^s$ consists of skew lines for $s=1,2$, respectively (cf. Corollary \ref{Cor:CiCjskew}). Finally, note that the statement (d) in Proposition \ref{PropIntersecoesLs} assures us that  $\rc^1\cup\rc^2$ also is formed by  skew lines. Therefore, $\rc^0\cup\rc^1\cup\rc^2$ consists of $3d$ skew lines in $\Fd$. 
\end{proof}

Below we will discuss some more examples that led us to believe that  ${\frak s}(\Fd)= 3d$ for $d\geq 7$ odd.

\begin{example}\label{ex:d=7} For $d=7$, let us consider $\rc^0=\{ L^0_{0,0}, L^0_{1,3}, L^0_{2,2}, L^0_{3,5}, L^0_{4,4}, L^0_{5,6}, L^0_{6,1}
\}$ which  consists of seven skew lines (cf. (a) in Corollary \ref{cor:Ciskew}). Furthermore, in the rows of the next table we register  the values of $\varphi_{7,\pm}(\rc^0)$, respectively.
\begin{center}
\begin{tabular}{|c||c|c|c|c|c|c|c|} \hline
         $\rc^0$    & $L^0_{0,0}$ & $L^0_{1,3}$  & $L^0_{2,2}$ & $L^0_{3,5}$ & $L^0_{4,4}$ & $L^0_{5,6}$ & $L^0_{6,1}$ \\ \hline \hline
          $\varphi_{7,-}$   & 0 & 2 & 0 & 2 & 0 & 1 & 2 \\ 
          \hline
          $\varphi_{7,+}$   & 0 & 4 & 4 & 1 & 1 & 4 & 0\\ \hline
        \end{tabular}
\end{center}

Now, having in mind  Corollary \ref{Cor:CiCjskew} for the choice of $\rc^s\subset\L^s$ such that $\rc^0\cup\rc^s$ is constituted by skew lines for $s=1,2$, it is necessary that 
$$
\rc^1\cap\L^2_k=\emptyset\quad\forall\, k\in\{0,1,2\}\quad\hbox{and}\quad \rc^2\cap\L^2_k=\emptyset\quad\forall\, k\in\{0,1,4\}.
$$ So,   $\rc^1=\{ L^1_{4,0},L^1_{4,2},L^1_{5,3},L^1_{5,4},L^1_{5,5},L^1_{6,1},L^1_{6,6}\}$ and $\rc^2=\{ L^2_{2,0},L^2_{2,5},L^2_{2,6},L^2_{3,1},L^2_{3,2},L^2_{3,3},L^2_{6,4}\}$ are admissible choices. As well, according to the information on the rows in the following two  tables:

\vspace{0.1cm}

\begin{center}
\begin{tabular}{|c||c|c|c|c|c|c|c|} \hline
         $\rc^1$    & $L^1_{4,0}$  & $L^1_{4,2}$ & $L^1_{5,3}$ & $L^1_{5,4}$ & $L^1_{5,5}$ & $L^1_{6,1}$ & $L^1_{6,6}$  \\ \hline \hline
          $\varphi_{7,+}$   & 4 & 6 & 1 & 2 & 3 & 0 & 5 \\
          \hline
          $\psi_7$   & 4 & 1 & 4 & 6 & 1 & 1 & 4 \\ \hline
        \end{tabular}	
\end{center}
\vspace{0.1cm}
\begin{center}
\begin{tabular}{|c||c|c|c|c|c|c|c|} \hline
         $\rc^2$    & $L^2_{2,0}$  & $L^2_{2,5}$ & $L^2_{2,6}$ & $L^2_{3,1}$  & $L^2_{3,2}$  & $L^2_{3,3}$ & $L^2_{6,4}$\\ \hline \hline
          $\varphi_{7,+}$   & 2 & 0 & 1 & 4 & 5 & 6 & 3 \\ 
          \hline
          $\psi_7$   & 2 & 5 & 0 & 5 & 0 & 2 & 0\\ \hline
        \end{tabular}	
\end{center}
\vspace{0.2cm}
we have that $\rc^s$ consists of skew lines for $s=1,2$ (cf. (b) and (c) in Corollary \ref{cor:Ciskew}) and $\psi_7(\rc^1)\cap \psi_7(\rc^2)=\emptyset$, which implies that $\rc^1\cup \rc^2$ also is formed by skew lines (cf. (c) in Corollary \ref{Cor:CiCjskew}). Therefore, $\rc:=\rc^0\cup\rc^1\cup\rc^2$ consists of 21 skew lines in ${\rm F}_7$. Thus, ${\frak s}({\rm F}_7)=21$ (since ${\frak s}({\rm F}_7)\leq 21$).
\end{example}

Let us go now to case $d=9$ and $d=11$.
\begin{example}\label{ex:d9ed11} The next tables contain  the necessary information  to conclude that those 27 lines (in ${\rm F}_9$) bellow are pairwise disjoint.

\begin{center}
\begin{tabular}{|c||c|c|c|c|c|c|c|c|c|} \hline
         $\rc^0$    & $L^0_{4,1}$ & $L^0_{5,2}$  & $L^0_{6,3}$ & $L^0_{7,7}$ & $L^0_{8,8}$ & $L^0_{0,0}$ & $L^0_{1,4}$  & $L^0_{2,5}$ & $L^0_{3,6}$\\ \hline \hline
          $\varphi_{9,-}$   & 6 & 6 & 6 & 0 & 0 & 0 & 3 & 3 & 3 \\ 
          \hline
          $\varphi_{9,+}$   & 5 & 7 & 0 & 5 & 7 & 0 & 5 & 7 & 0\\ \hline
        \end{tabular}
\end{center}
\begin{center}
\begin{tabular}{|c||c|c|c|c|c|c|c|c|c|} \hline
         $\rc^1$    & $L^1_{5,0}$  & $L^1_{5,1}$ & $L^1_{1,2}$ & $L^1_{1,3}$ & $L^1_{7,4}$ & $L^1_{5,5}$ & $L^1_{2,6}$ & $L^1_{2,7}$ & $L^1_{8,8}$  \\ \hline \hline
          $\varphi_{9,+}$   & 5 & 6 & 3 & 4 & 2 & 1 & 8 & 0 & 7 \\
          \hline
          $\psi_9$   & 5 & 7 & 5 & 7 & 6 & 6 & 5 & 7 & 6 \\ \hline
        \end{tabular}	
\end{center}

\begin{center}
\begin{tabular}{|c||c|c|c|c|c|c|c|c|c|} \hline
         $\rc^2$    & $L^2_{1,0}$  & $L^2_{1,1}$ & $L^2_{6,2}$ & $L^2_{6,3}$  & $L^2_{2,4}$  & $L^2_{2,5}$ & $L^2_{6,6}$ &  $L^2_{6,7}$ & $L^2_{6,8}$ \\ \hline \hline
          $\varphi_{9,+}$   & 1 & 2 & 8 & 0 & 6 & 7 & 3 & 4 & 5\\ 
          \hline
          $\psi_9$   & 1 & 3 & 1 & 3 & 1 & 3 & 0 & 2 & 4\\ \hline
        \end{tabular}	
\end{center}

Now, we show the tables for the  lines in ${\rm F}_{11}$
\begin{center}
\begin{tabular}{|c||c|c|c|c|c|c|c|c|c|c|c|} \hline
         $\rc^0$    & $L^0_{5,0}$ & $L^0_{6,1}$  & $L^0_{7,2}$ & $L^0_{8,3}$ & $L^0_{9,7}$ & $L^0_{10,8}$ & $L^0_{0,9}$  & $L^0_{1,4}$ & $L^0_{2,5}$ & $L^0_{3,6}$ & $L^0_{4,10}$\\ \hline \hline
          $\varphi_{11,-}$   & 6 & 6 & 6 & 6 & 9 & 9 & 9 & 3 & 3 & 3 & 6\\ 
          \hline
          $\varphi_{11,+}$   & 5 & 7 & 9 & 0 & 5 & 7 & 9 & 5 & 7 & 9 & 3\\ \hline
        \end{tabular}
\end{center}
\begin{center}
\begin{tabular}{|c||c|c|c|c|c|c|c|c|c|c|c|} \hline
         $\rc^1$    & $L^1_{8,0}$  & $L^1_{8,1}$ & $L^1_{4,2}$ & $L^1_{4,3}$ & $L^1_{0,4}$ & $L^1_{0,5}$ & $L^1_{7,6}$ & $L^1_{7,7}$ & $L^1_{4,8}$& $L^1_{1,9}$ & $L^1_{1,10}$  \\ \hline \hline
          $\varphi_{11,+}$   & 8 & 9 & 6 & 7 & 4 & 5 & 2 & 3 & 1 & 10 & 0 \\
          \hline
          $\psi_{11}$   & 8 & 10 & 8 & 10 & 8 & 10 & 8 & 10 & 9 & 8 & 10 \\ \hline
        \end{tabular}	
\end{center}

\begin{center}
\begin{tabular}{|c||c|c|c|c|c|c|c|c|c|c|c|} \hline
         $\rc^2$    & $L^2_{1,0}$  & $L^2_{1,1}$ & $L^2_{1,2}$ & $L^2_{6,3}$  & $L^2_{6,4}$  & $L^2_{6,5}$ & $L^2_{2,6}$ &  $L^2_{8,7}$ & $L^2_{8,8}$ &  $L^2_{8,9}$ & $L^2_{8,10}$  \\ \hline \hline
          $\varphi_{11,+}$   & 1 & 2 & 3 & 9 & 10 & 0 & 8 & 4 & 5 & 6 & 7\\ 
          \hline
          $\psi_{11}$   & 1 & 3 & 5 & 1 & 3 & 5 & 3 & 0 & 2 & 4 & 6\\ \hline
        \end{tabular}	
\end{center}

\end{example}

In Propositions \ref{Prop:dimparnpar} and \ref{Prop:dimparnimpar}, the indices $a,b$ in the notation $L_{a,b}^s$ are always to be considered modulo $d$.

\begin{proposition}
\label{Prop:dimparnpar} Let $d=2n+1$ with $n=2k$ and $k\geq 3$. Consider the families 
$$
\rc^0=
\Big\{ L^0_{1+i, 2k+i}\Big\}_{i=0}^k
\cup 
\Big\{L^0_{k+i,k+i}\Big\}_{i=2}^{k-1}
\cup
\Big\{ L^0_{2k+i,1+i}\Big\}_{i=0}^k
\cup 
\Big\{ L^0_{3k+i,3k+i}\Big\}_{i=1}^{k+1},
$$
$$
\rc^1=\Big\{ L^1_{2k+1,2k+i}\Big\}_{i=1}^{k-1}
\cup 
\Big\{ L^1_{3k+1,2k}\Big\}
\cup 
 \Big\{ L^1_{1,k+i}\Big\}_{i=0}^{k-1}
 \cup
 \Big\{ L^1_{2k+1,i}\Big\}_{i=0}^{k-1}
 \cup 
\Big\{ L^1_{3k+2,4k}\Big\}
\cup
 \Big\{ L^1_{2, 3k+i}\Big\}_{i=0}^{k-1},
$$
$$
\rc^2=\Big\{ L^2_{1,i}\Big\}_{i=0}^{k-1}
\cup 
\Big\{ L^2_{2k+2,3k+i}\Big\}_{i=0}^{k}
\cup 
 \Big\{ L^2_{2,2k+i}\Big\}_{i=0}^{k-1}
 \cup
 \Big\{ L^2_{2k+2,k+i}\Big\}_{i=0}^{k-1}.
$$ It is verified that $\rc^0\cup\rc^1\cup\rc^2$ consists of $3d$ skew lines in ${\rm F}_d$.
\end{proposition}
\begin{proof} Let us devide the proof into four steps.
\\
[1mm]
{\sf Step 1: $\rc^0$  is constituted by $d$ skew lines.}

First of all, note that $\rc^0$ is defined by four strata below 
$$
\rc^0=\underbrace{
\Big\{ L^0_{1+i, 2k+i}\Big\}_{i=0}^k}_{({\rm i})}
\cup 
\underbrace{\Big\{L^0_{k+i,k+i}\Big\}_{i=2}^{k-1}}_{({\rm ii})}
\cup
\underbrace{\Big\{ L^0_{2k+i,1+i}\Big\}_{i=0}^k}_{({\rm iii})}
\cup 
\underbrace{\Big\{ L^0_{3k+i,3k+i}\Big\}_{i=1}^{k+1}}_{({\rm iv})},
$$ where the stratum (ii) is non-empty if and only if $k\geq 3$ (so, $d\neq 5$ and $d\neq 9$). Furthermore, we have that the label $t$ in each $L^0_{t,j}\in \rc^0$ is varying throughout the set 
\begin{equation}\label{eq:tC0}
\{\underbrace{1,\ldots,k+1}_{({\rm i})},\underbrace{k+2,\ldots,2k-1}_{({\rm ii})},\underbrace{2k,\ldots,3k}_{({\rm iii})},\underbrace{3k+1,\ldots,4k,4k+1\equiv_d 0}_{({\rm iv})}\}.
\end{equation} And the label $j$ throughout  the set
\begin{equation}\label{eq:jC0}
\{\underbrace{1,\ldots,k+1}_{({\rm iii})},\underbrace{k+2,\ldots,2k-1}_{({\rm ii})},\underbrace{2k,\ldots,3k}_{({\rm i})},\underbrace{3k+1,\ldots,4k,4k+1\equiv_d 0}_{({\rm iv})}\}.
\end{equation} Since the sets in (\ref{eq:tC0}) and (\ref{eq:jC0}) are equal to $R_d$, it follows that $\rc^0$ is constituted by $d$ skew lines. 
\\
[1mm]
{\sf Step 2: $\rc^0\cup \rc^s$  is constituted by skew lines for $s=1,2$.}

Next, we display the values of $\varphi_{d,\pm}$ over $\rc^0$ (using the stratification (i),\ldots,(iv) for $\rc^0$).

\begin{center}
\begin{tabular}{|c||c|c|c|c|c|c|c|} \hline
         $\rc^0$/ (i)    & $L^0_{1,2k}$ & $L^0_{2,2k+1}$  & $L^0_{3,2k+2}$ & $\cdots$ & $L^0_{k-1,3k-2}$  & $L^0_{k,3k-1}$ & $L^0_{1+k,3k}$\\ \hline \hline
          $\varphi_{d,-}$   & $2k-1$ & $2k-1$ & $2k-1$ & $\cdots$ & $2k-1$ & $2k-1$ & $2k-1$ \\ 
          \hline
          $\varphi_{d,+}$   & $2k+1$ & $2k+3$ & $2k+5$ & $\cdots$ & $4k-3$ & $4k-1$ & $4k+1\equiv_d 0$\\ \hline
        \end{tabular}
\end{center}

\begin{center}
\begin{tabular}{|c||c|c|c|c|c|c|c|} \hline
         $\rc^0$/ (ii)    & $L^0_{k+2,k+2}$ & $L^0_{k+3,k+3}$  & $L^0_{k+4,k+4}$ & $\cdots$ & $L^0_{2k-3,2k-3}$  & $L^0_{2k-2,2k-2}$ & $L^0_{2k-1,2k-1}$\\ \hline \hline
          $\varphi_{d,-}$   & $0$ & $0$ & $0$ & $\cdots$ & $0$ & $0$ & $0$ \\ 
          \hline
          $\varphi_{d,+}$   & $2k+4$ & $2k+6$ & $2k+8$ & $\cdots$ & $4k-6$ & $4k-4$ & $4k-2$\\ \hline
        \end{tabular}
\end{center}

\begin{center}
\begin{tabular}{|c||c|c|c|c|c|c|c|} \hline
         $\rc^0$/ (iii)    & $L^0_{2k,1}$ & $L^0_{2k+1,2}$  & $L^0_{2k+2,3}$ & $\cdots$ & $L^0_{3k-2,k-1}$  & $L^0_{3k-1,k}$ & $L^0_{3k,1+k}$\\ \hline \hline
          $\varphi_{d,-}$   & $2k+2$ & $2k+2$ & $2k+2$ & $\cdots$ & $2k+2$ & $2k+2$ & $2k+2$ \\ 
          \hline
          $\varphi_{d,+}$   & $2k+1$ & $2k+3$ & $2k+5$ & $\cdots$ & $4k-3$ & $4k-1$ & $0$\\ \hline
        \end{tabular}
\end{center}

\begin{center}
\begin{tabular}{|c||c|c|c|c|c|c|c|} \hline
         $\rc^0$/ (iv)    & $L^0_{3k+1,3k+1}$ & $L^0_{3k+2,3k+2}$  & $L^0_{3k+3,3k+3}$ & $\cdots$ & $L^0_{4k-1,4k-1}$  & $L^0_{4k,4k}$ & $L^0_{0,0}$\\ \hline \hline
          $\varphi_{d,-}$   & $0$ & $0$ & $0$ & $\cdots$ & $0$ & $0$ & $0$ \\ 
          \hline
          $\varphi_{d,+}$   & $2k+1$ & $2k+3$ & $2k+5$ & $\cdots$ & $4k-3$ & $4k-1$ & $0$\\ \hline
        \end{tabular}
\end{center}

Now, having in mind  Corollary \ref{Cor:CiCjskew} for the choice of $\rc^s\subset\L^s$ such that $\rc^0\cup\rc^s$ is constituted by skew lines for $s=1,2$ and the tables (involving $\rc^0$) above, it is necessary that 
\begin{equation}\label{eq:intCsLsknpar}
\left\{
\begin{array}{l}
\rc^1\cap\L^1_t=\emptyset\quad\forall\,\, t\in\{0,2k-1,2k+2\}\\
\vspace{-0.4cm}\\
\rc^2\cap\L^2_t=\emptyset\quad\forall\,\, t\in\{2k+j\}_{j=1}^{2k+1}\--\{2k+2,4k\}.
\end{array}
\right.
\end{equation}
Now, it is a straightforward verification to see that the label $t$ in each $L^s_{t,j}\in \rc^s$ belongs to the set  
$$\{1,2,2k+1,3k+1,3k+2\}\,\hbox{ for }\, s=1\,\hbox{ and }\, \{1,2,2k+2\} \,\hbox{ for }\, s=2.$$ 
Thus, using (\ref{eq:intCsLsknpar}) we concluded that $\rc^0\cup\rc^s$ is constituted by skew lines for $s=1,2$.
\\
[1mm]
{\sf Step 3: $\rc^s$  is constituted by $d$ skew lines for $s=1,2$.}

Let us stratify $\rc^1$ as follows: $\rc^1={\rm A}_1\,\dot\cup\,{\rm A}_2\,\dot\cup\,
{\rm A}_3\,\dot\cup\,
{\rm A}_4\,\dot\cup\,
{\rm A}_5\,\dot\cup\,
{\rm A}_6
$ where
\begin{equation}\label{eq:C1As}
 \begin{array}{c}
 {\rm A}_1 :=\Big\{ L^1_{2k+1,2k+i}\Big\}_{i=1}^{k-1},\,\, 
{\rm A}_2 :=
\Big\{ L^1_{3k+1,2k}\Big\},\,\,
{\rm A}_3 :=
 \Big\{ L^1_{1,k+i}\Big\}_{i=0}^{k-1},\\
 \vspace{-0.3cm}\\
 {\rm A}_4 :=
 \Big\{ L^1_{2k+1,i}\Big\}_{i=0}^{k-1},\,\,
 {\rm A}_5 := 
\Big\{ L^1_{3k+2,4k}\Big\},\,\,
{\rm A}_6 :=
 \Big\{ L^1_{2, 3k+i}\Big\}_{i=0}^{k-1} .
 \end{array}   
\end{equation}

Note that the label $j$ in each $L^1_{t,j}\in \rc^1$ is varying throughout the set 
\begin{equation}\label{eq:jC1}
\{\underbrace{0,\ldots,k-1}_{{\rm A}_4},\underbrace{k,\ldots,2k-1}_{{\rm A}_3},\underbrace{2k}_{{\rm A}_2},\underbrace{2k+1,\ldots,3k-1}_{{\rm A}_1},\underbrace{3k,\ldots,4k-1}_{{\rm A}_6},\underbrace{4k}_{{\rm A}_5} \},
\end{equation}
which is equal to $R_d$. Furthermore, $\varphi_{d,+}(\rc^1)$ is given by

\vspace{0.2cm}
\noindent
\begin{tabular}{|c||c|c|c|c|c|c|c|} \hline
         ${\rm A}_1 $   & $L^1_{2k+1,2k+1}$ & $L^1_{2k+1,2k+2}$  & $L^1_{2k+1,2k+3}$ & $\cdots$ & $L^1_{2k+1,3k-3}$  & $L^1_{2k+1,3k-2}$ & $L^1_{2k+1,3k-1}$\\ \hline \hline
          $\varphi_{d,+}$   & $1$ & $2$ & $3$ & $\cdots$ & $k-3$ & $k-2$ & $k-1$ \\\hline
        \end{tabular}

\begin{center}
\begin{tabular}{|c||c|} \hline
         ${\rm A}_2$    & $L^1_{3k+1,2k}$\\ \hline \hline
          $\varphi_{d,+}$   & $k$  \\
          \hline
        \end{tabular}
\end{center}

\begin{center}
\begin{tabular}{|c||c|c|c|c|c|c|c|} \hline
         ${\rm A}_3$    & $L^1_{1,k}$ & $L^1_{1,k+1}$  & $L^1_{1,k+2}$ & $\cdots$ & $L^1_{1,2k-3}$  & $L^1_{1,2k-2}$ & $L^1_{1,2k-1}$\\ \hline \hline
          $\varphi_{d,+}$   & $k+1$ & $k+2$ & $k+3$ & $\cdots$ & $2k-2$ & $2k-1$ & $2k$ \\\hline
        \end{tabular}
\end{center}

\begin{center}
\begin{tabular}{|c||c|c|c|c|c|c|c|} \hline
         ${\rm A}_4$    & $L^1_{2k+1,0}$ & $L^1_{2k+1,1}$  & $L^1_{2k+1,2}$ & $\cdots$ & $L^1_{2k+1,k-3}$  & $L^1_{2k+1,k-2}$ & $L^1_{2k+1,k-1}$\\ \hline \hline
          $\varphi_{d,+}$   & $2k+1$ & $2k+2$ & $2k+3$ & $\cdots$ & $3k-2$ & $3k-1$ & $3k$ \\\hline
        \end{tabular}
\end{center}

\begin{center}
\begin{tabular}{|c||c|} \hline
         ${\rm A}_5$    & $L^1_{3k+2,4k}$\\ \hline \hline
          $\varphi_{d,+}$   & $3k+1$  \\
          \hline
        \end{tabular}
\end{center}

\begin{center}
\begin{tabular}{|c||c|c|c|c|c|c|c|} \hline
         ${\rm A}_6$    & $L^1_{2,3k}$ & $L^1_{2,3k+1}$  & $L^1_{2,3k+2}$ & $\cdots$ & $L^1_{2,4k-3}$  & $L^1_{2,4k-2}$ & $L^1_{2,4k-1}$\\ \hline \hline
          $\varphi_{d,+}$   & $3k+2$ & $3k+3$ & $3k+4$ & $\cdots$ & $4k-1$ & $4k$ & $0$ \\\hline
        \end{tabular}
\end{center}

Thus, $\varphi_{d,+}(\rc^1)=R_d$. Taking into account the established facts, we may use Corollary \ref{cor:Ciskew} to conclude that $\rc^1$  is constituted by $d$ skew lines.

\vspace{0.2cm}

In a similar way, let us consider the following stratification for $\rc^2$: $\rc^2={\rm B}_1\,\dot\cup\,{\rm B}_2\,\dot\cup\,
{\rm B}_3\,\dot\cup\,
{\rm B}_4
$ where
\begin{equation}\label{eq:C2Bs}
\begin{array}{c}
{\rm B}_1 :=\Big\{ L^2_{1,i}\Big\}_{i=0}^{k-1},\,\, 
{\rm B}_2 :=
\Big\{ L^2_{2k+2,3k+i}\Big\}_{i=0}^k,\,\,
\\
\vspace{-0.3cm}\\
{\rm B}_3 :=
 \Big\{ L^2_{2,2k+i}\Big\}_{i=0}^{k-1},\,\,
{\rm B}_4 :=
 \Big\{ L^2_{2k+2,k+i}\Big\}_{i=0}^{k-1}.
 \end{array}
 \end{equation}

Note that the label $j$ in each $L^2_{t,j}\in \rc^2$ is varying throughout the set 
\begin{equation}\label{eq:jC2}
\{\underbrace{0,\ldots,k-1}_{{\rm B}_1},\underbrace{k,\ldots,2k-1}_{{\rm B}_4},\underbrace{2k,\ldots,3k-1}_{{\rm B}_3},\underbrace{3k,\ldots,4k}_{{\rm B}_2}\},
\end{equation}
which is equal to $R_d$. Furthermore, $\varphi_{d,+}(\rc^2)$ is given by

\begin{center}
\begin{tabular}{|c||c|c|c|c|c|c|c|} \hline
         ${\rm B}_1$    & $L^2_{1,0}$ & $L^2_{1,1}$  & $L^2_{1,2}$ & $\cdots$ & $L^2_{1,k-3}$  & $L^2_{1,k-2}$ & $L^2_{1,k-1}$\\ \hline \hline
          $\varphi_{d,+}$   & $1$ & $2$ & $3$ & $\cdots$ & $k-2$ & $k-1$ & $k$ \\\hline
        \end{tabular}
\end{center}

\begin{center}
\begin{tabular}{|c||c|c|c|c|c|c|c|} \hline
         ${\rm B}_2$    & $L^2_{2k+2,3k}$ & $L^2_{2k+2,3k+1}$  & $L^2_{2k+2,3k+2}$ & $\cdots$ & $L^2_{2k+2,4k-2}$  & $L^2_{2k+2,4k-1}$ & $L^2_{2k+2,4k}$\\ \hline \hline
          $\varphi_{d,+}$   & $k+1$ & $k+2$ & $k+3$ & $\cdots$ & $2k-1$ & $2k$ & $2k+1$ \\\hline
        \end{tabular}
\end{center}

\begin{center}
\begin{tabular}{|c||c|c|c|c|c|c|c|} \hline
         ${\rm B}_3$    & $L^2_{2,2k}$ & $L^2_{2,2k+1}$  & $L^2_{2,2k+2}$ & $\cdots$ & $L^2_{2,3k-3}$  & $L^2_{2,3k-2}$ & $L^2_{2,3k-1}$\\ \hline \hline
          $\varphi_{d,+}$   & $2k+2$ & $2k+3$ & $2k+4$ & $\cdots$ & $3k-1$ & $3k$ & $3k+1$ \\\hline
        \end{tabular}
\end{center}

\begin{center}
\begin{tabular}{|c||c|c|c|c|c|c|c|} \hline
         ${\rm B}_4$    & $L^2_{2k+2,k}$ & $L^2_{2k+2,k+1}$  & $L^2_{2k+2,k+2}$ & $\cdots$ & $L^2_{2k+2,2k-3}$  & $L^2_{2k+2,2k-2}$ & $L^2_{2k+2,2k-1}$\\ \hline \hline
          $\varphi_{d,+}$   & $3k+2$ & $3k+3$ & $3k+4$ & $\cdots$ & $4k-1$ & $4k$ & $0$ \\\hline
        \end{tabular}
\end{center}

Thus, $\varphi_{d,+}(\rc^2)=R_d$. Again, using  Corollary \ref{cor:Ciskew}, we concluded that $\rc^2$  is constituted by $d$ skew lines.
\\
[1mm]
{\sf Step 4: $\rc^1\cup \rc^2$  is constituted by $2d$ skew lines.}

Having in mind  (c) in Corollary \ref{Cor:CiCjskew}, it is enough to prove that $\psi_d(\rc^1)\,\cap\, \psi_d(\rc^2)=\emptyset$. So, we will use again the stratification for $\rc^1$ in (\ref{eq:C1As}) and $\rc^2$ in (\ref{eq:C2Bs})  to display the computation of $\psi_d(\rc^1)$ and $\psi_d(\rc^2)$ bellow:

\begin{center}
\noindent
\begin{tabular}{|c||c|c|c|c|c|c|c|} \hline
         ${\rm A}_1 $   & $L^1_{2k+1,2k+1}$ & $L^1_{2k+1,2k+2}$  & $L^1_{2k+1,2k+3}$ & $\cdots$ & $L^1_{2k+1,3k-3}$  & $L^1_{2k+1,3k-2}$ & $L^1_{2k+1,3k-1}$\\ \hline \hline
          $\psi_{d}$   & $2k+2$ & $2k+4$ & $2k+6$ & $\cdots$ & $4k-6$ & $4k-4$ & $4k-2$ \\\hline
        \end{tabular}
\end{center}

\begin{center}
\begin{tabular}{|c||c|} \hline
         ${\rm A}_2$    & $L^1_{3k+1,2k}$\\ \hline \hline
          $\psi_{d}$   & $3k$  \\
          \hline
        \end{tabular}
\end{center}

\begin{center}
\begin{tabular}{|c||c|c|c|c|c|c|c|} \hline
         ${\rm A}_3$    & $L^1_{1,k}$ & $L^1_{1,k+1}$  & $L^1_{1,k+2}$ & $\cdots$ & $L^1_{1,2k-3}$  & $L^1_{1,2k-2}$ & $L^1_{1,2k-1}$\\ \hline \hline
          $\psi_{d}$   & $2k+1$ & $2k+3$ & $2k+5$ & $\cdots$ & $4k-5$ & $4k-3$ & $4k-1$ \\\hline
        \end{tabular}
\end{center}

\begin{center}
\begin{tabular}{|c||c|c|c|c|c|c|c|} \hline
         ${\rm A}_4$    & $L^1_{2k+1,0}$ & $L^1_{2k+1,1}$  & $L^1_{2k+1,2}$ & $\cdots$ & $L^1_{2k+1,k-3}$  & $L^1_{2k+1,k-2}$ & $L^1_{2k+1,k-1}$\\ \hline \hline
          $\psi_{d}$   & $2k+1$ & $2k+3$ & $2k+5$ & $\cdots$ & $4k-5$ & $4k-3$ & $4k-1$ \\\hline
        \end{tabular}
\end{center}

\begin{center}
\begin{tabular}{|c||c|} \hline
         ${\rm A}_5$    & $L^1_{3k+2,4k}$\\ \hline \hline
          $\psi_{d}$   & $3k$  \\
          \hline
        \end{tabular}
\end{center}

\begin{center}
\begin{tabular}{|c||c|c|c|c|c|c|c|} \hline
         ${\rm A}_6$    & $L^1_{2,3k}$ & $L^1_{2,3k+1}$  & $L^1_{2,3k+2}$ & $\cdots$ & $L^1_{2,4k-3}$  & $L^1_{2,4k-2}$ & $L^1_{2,4k-1}$\\ \hline \hline
          $\psi_{d}$   & $2k+1$ & $2k+3$ & $2k+5$ & $\cdots$ & $4k-5$ & $4k-3$ & $4k-1$ \\\hline
        \end{tabular}
\end{center}

So,
\begin{equation}\label{eq:psiC1npar}
\psi_d(\rc^1)=\{2k+1,2k+2,\ldots, 4k-2,4k-1\}.    
\end{equation}

\begin{center}
\begin{tabular}{|c||c|c|c|c|c|c|c|} \hline
         ${\rm B}_1$    & $L^2_{1,0}$ & $L^2_{1,1}$  & $L^2_{1,2}$ & $\cdots$ & $L^2_{1,k-3}$  & $L^2_{1,k-2}$ & $L^2_{1,k-1}$\\ \hline \hline
          $\psi_{d}$   & $1$ & $3$ & $5$ & $\cdots$ & $2k-5$ & $2k-3$ & $2k-1$ \\\hline
        \end{tabular}
\end{center}

\begin{center}
\begin{tabular}{|c||c|c|c|c|c|c|c|} \hline
         ${\rm B}_2$    & $L^2_{2k+2,3k}$ & $L^2_{2k+2,3k+1}$  & $L^2_{2k+2,3k+2}$ & $\cdots$ & $L^2_{2k+2,4k-2}$  & $L^2_{2k+2,4k-1}$ & $L^2_{2k+2,4k}$\\ \hline \hline
          $\psi_{d}$   & $0$ & $2$ & $4$ & $\cdots$ & $2k-4$ & $2k-2$ & $2k$ \\\hline
        \end{tabular}
\end{center}

\begin{center}
\begin{tabular}{|c||c|c|c|c|c|c|c|} \hline
         ${\rm B}_3$    & $L^2_{2,2k}$ & $L^2_{2,2k+1}$  & $L^2_{2,2k+2}$ & $\cdots$ & $L^2_{2,3k-3}$  & $L^2_{2,3k-2}$ & $L^2_{2,3k-1}$\\ \hline \hline
          $\psi_{d}$   & $1$ & $3$ & $5$ & $\cdots$ & $2k-5$ & $2k-3$ & $2k-1$ \\\hline
        \end{tabular}
\end{center}

\begin{center}
\begin{tabular}{|c||c|c|c|c|c|c|c|} \hline
         ${\rm B}_4$    & $L^2_{2k+2,k}$ & $L^2_{2k+2,k+1}$  & $L^2_{2k+2,k+2}$ & $\cdots$ & $L^2_{2k+2,2k-3}$  & $L^2_{2k+2,2k-2}$ & $L^2_{2k+2,2k-1}$\\ \hline \hline
          $\psi_{d}$   & $1$ & $3$ & $5$ & $\cdots$ & $2k-5$ & $2k-3$ & $2k-1$ \\\hline
        \end{tabular}
\end{center}
Therefore,
\begin{equation}\label{eq:psiC2npar}
\psi_d(\rc^2)=\{0,1,2,\ldots, 2k-1,2k\}.   \end{equation}
Thus, from (\ref{eq:psiC1npar}) and \ref{eq:psiC2npar}) we have that $\psi_d(\rc^1)\,\cap\, \psi_d(\rc^2)=\emptyset$.
\end{proof}

\begin{proposition}\label{Prop:dimparnimpar} Let $d=2n+1$ with $n=2k+ 1$ and $k\geq 3$. Consider the families 
 $$
	\rc^0= \big\{ L^0_{2k+1+i, 0}\big\}_{i=0}^{k+1}
\cup 
\big\{L^0_{k+1+i,3k+i}\big\}_{i=1}^{k+2}
\cup
\big\{ L^0_{2+i,2k-1+i}\big\}_{i=0}^{k+1} \cup\big\{ L^0_{k+4+i,k+2+i}\big\}_{i=0}^{k-4}, $$ $$
\rc^1= \big\{ L^1_{2k+1,2k+i}\big\}_{i=1}^{k-1}
\cup 
\big\{ L^1_{3k+1,2k}\big\}
\cup 
 \big\{ L^1_{1,k+i}\big\}_{i=0}^{k-1}
 \cup
 \big\{ L^1_{2k+1,i}\big\}_{i=0}^{k-1}  \cup\big\{ L^1_{3k+2,4k}\big\}
\cup
 \big\{ L^1_{2, 3k+i}\big\}_{i=0}^{k-1}, $$ $$
 \rc^2=\ \big\{ L^2_{3,i}\big\}_{i=0}^{k}
\cup 
\big\{ L^2_{2k+4,k+1+i}\big\}_{i=0}^{k}
\cup 
 \big\{ L^2_{2,2k+2+i}\big\}_{i=0}^{k}
 \cup
 \big\{ L^2_{2k+4,3k+3+i}\big\}_{i=0}^{k-1}.
$$
It is verified that $\rc^0\cup\rc^1\cup\rc^2$ consists of $3d$ skew lines in ${\rm F}_d$.
\end{proposition}

\begin{proof} It is analogous to the proof presented in the Proposition~\ref{Prop:dimparnpar}. \end{proof}

It follows from the previous Propositions~\ref{Prop:dimparnpar}, \ref{Prop:dimparnimpar} and the Examples~\ref{ex:d=7},~\ref{ex:d9ed11} that

\begin{corollary}Assume $d\geq 7$ odd. Then ${\frak s}(F_d)=3d$.   
\end{corollary}


\begin{theorem}\label{Teo} Let $\Fd$ be the Fermat surface of degree $d\geq 3$. If $\frak{s} (\Fd)$ is the maximal number of skew lines in $\Fd$, then $\frak{s}(\Fd)=3d$ for all $d\not=3,5.$ Being $\frak{s}(\textrm{F}_3)=6$ and  $ {\frak s}({\rm F}_5)=13.$\end{theorem}

\begin{proof} For $d\in \{3,5\}$ see Section~\ref{Secd35}. For the other cases, see the previous results in this section.\end{proof}


\vspace{0.3cm}

\begin{center}
    
\footnotesize{\begin{tabular}{l l l }
  Sally ANDRIA & Jacqueline ROJAS & Wállace MANGUEIRA\\
  Dep. of Geometry, UFF & Dep. of Mathematics, UFPB & Dep. of Mathematics, UFPB\\
   sally\_andria@id.uff.br & jfra@academico.ufpb.br & wallace.mangueira@academico.ufpb.br 
\end{tabular}}

\end{center}

\end{document}